\documentclass[a4paper,12pt]{article}

\usepackage{amsfonts,amssymb,amsmath,dsfont,relsize,accents}
\usepackage{amsthm}
\usepackage[english]{babel}
\usepackage[small,labelfont=it,margin=1em]{caption}
\usepackage{color}
\usepackage[utf8]{inputenc}
\usepackage{latexsym}
\usepackage{mathrsfs}
\usepackage{upgreek}
\usepackage{graphicx}
\usepackage{hyperref}
\usepackage[normalem]{ulem}

\DeclareMathAlphabet{\mathdutchcal}{U}{dutchcal}{m}{n}
\SetMathAlphabet{\mathdutchcal}{bold}{U}{dutchcal}{b}{n}
\DeclareMathAlphabet{\mathdutchbcal}{U}{dutchcal}{b}{n}
\newcommand*{\s}{{\mathdutchcal{s}}}

\newtheorem{theorem}{Theorem}[section]
\newtheorem*{theorem*}{Theorem}
\newtheorem{lemma}{Lemma}[section]

\theoremstyle{definition}
\newtheorem*{definition*}{Definition}

\newcommand*{\E}{\mathbb{E}}

\newcommand*{\N}{\mathbb{N}}
\renewcommand{\P}{\mathbb{P}}

\newcommand*{\T}{\mathbb{T}}
\newcommand*{\V}{\mathbb{V}}
\newcommand*{\Z}{\mathbb{Z}}

\def \uoms {{\bar{\omega}_\s}}
\def \uoml {{\bar{\omega}}_{\ell}}

\renewcommand{\leq}{\leqslant}
\renewcommand{\le}{\leqslant}
\renewcommand{\geq}{\geqslant}
\renewcommand{\ge}{\geqslant}
\renewcommand{\subset}{\subseteq}
\renewcommand{\supset}{\supseteq}

\DeclareMathOperator{\prog}{prog}

\hypersetup{colorlinks=true,urlcolor=blue,linkcolor=black,citecolor=black}

\newcommand*{\doi}[1]{\href{http://dx.doi.org/\detokenize{#1}}{doi}}

\renewcommand{\baselinestretch}{1.15}

\begin{document}

\title{Monotonicity and phase diagram for multi-range percolation on oriented trees}
\author{Bernardo N. B. de Lima\footnote{UFMG-Federal University of Minas Gerais, Brazil},
{}
Leonardo T. Rolla\footnote{Argentinian National Research Council at the University of Buenos Aires},\footnote{NYU-ECNU Institute of Mathematical Sciences at NYU Shanghai, China}
{}
Daniel Valesin\footnote{University of Groningen, The Netherlands}
}
\maketitle

\begin{abstract}
We consider Bernoulli bond percolation on oriented regular trees, where besides the usual short bonds, all bonds of a certain length are added. Independently, short bonds are open with probability $p$ and long bonds are open with probability $q$. We study properties of the critical curve which delimits the set of pairs $(p,q)$ for which there are almost surely no infinite paths.
We also show that this curve decreases with respect to the length of the long bonds.
\end{abstract}

\section{Introduction}

Consider the graph having $\Z^d$ as vertex set and all edges of the form $\{x, x\pm e_i\}$ and $\{x,x\pm k\cdot e_i\}$ for some $k\ge 2$.
It was shown in~\cite{LimaSanchisSilva11} that the critical probability for Bernoulli bond percolation on this graph converges to that of $\Z^{2d}$ as $k \to \infty$.
This result, later generalized in~\cite{MartineauTassion17}, is a particular instance of Schramm's conjecture~\cite{BenjaminiNachmiasPeres11} that the percolation threshold for transitive graphs is a local property.
The convergence proved in~\cite{LimaSanchisSilva11} is conjectured to be monotone, that is, the percolation threshold for the above graph should be decreasing in the length $k$ of long edges.\footnote
{In support of this conjecture, simulations~\cite{AtmanSchnabelYY} confirm that increasing $k$ decreases the critical parameter, and the proof of~\cite[Lemma~2]{LimaSanchisSilva11} shows that replacing $k$ by a multiple of $k$ does not increase it.}

Monotonicity questions are often intriguing for being extremely simple to ask and hard to answer.
A good example~\cite{Berg07} is the following: for Bernoulli bond percolation on the usual graph $\Z^d$, prove that the probability of the origin being connected to $(n,0,\dots,0)$ monotone in $n$.
This problem is still open, except when the parameter is close to $0$ or $1$~\cite{LimaProcacciSanchis15}.
For oriented percolation on $\Z^2_+$, the probability of the origin being connected to $(m-n, m+n)$ is decreasing in $n\in\{0,\dots,m\}$ for fixed $m$; this may be obvious but the proof is not straightforward~\cite{AndjelSued08}.
In the same spirit, for unoriented percolation on $\Z^2_+$, if the parameter is smaller for horizontal edges than for vertical ones, the above probability should be larger than the probability of the origin being connected to $(m+n, m-n)$. This has only been proved under the assumption that the ratio between horizontal and vertical parameters is small enough~\cite{DeLimaProcacci04}.

For first-passage percolation, it was conjectured~\cite{HammersleyWelsh65} that the expected minimum travel time from $(0,0)$ to $(n,0)$ along paths contained in the
strip $\{(x,y):0\leq x\leq n\}$ is nondecreasing in $n$.
This question is still open, with a number of partial results~\cite{Ahlberg15,AlmWierman99,Howard01,Gouere14}.
In the negative direction, for first-passage percolation on $\Z_+\times\Z$, there is a counter-example~\cite{Berg83}
where the expected passage time from the origin to $(2,0)$ is less than the expected passage time from the origin to $(1,0)$.
Another context where strict monotonicity is expected to happen is in the case of essential enhancements as introduced in~\cite{AizenmanGrimmett91}, see also~\cite{BalisterBollobasRiordan14}.

In this paper we consider percolation on $\T_{d,k}$, the graph given by the oriented rooted $d$-ary tree ($d \ge 2$) with a root at the top, bearing the usual ``short'' downward edges plus the addition of all downward edges of length $k$, called ``long'' edges.
This is an oriented version of Trofimov's grandfather graph for $k=2$, or the great$^k$-grandfather graph for larger $k$.
We let short and long edges be open independently with probability $p$ and $q$, respectively.
The phase space $[0,1]^2$ is decomposed in two regions: a set $\mathcal{P}_k$ of pairs $(p,q)$ for which a.s.\ there are infinite open paths, and a set $\mathcal{N}_k$ of pairs for which a.s.\ there are none, see Figure~\ref{fig:phasespaceab}a.

For $p>\frac{1}{d}$ there are a.s.\ infinite open paths of short edges, and for $q>\frac{1}{d^k}$ there are a.s.\ infinite open paths of long edges.
For $dp +d^k q \leq 1$, a simple comparison with a branching process (to be given in \S\ref{sec:pcqc}) shows that a.s.\ there are no infinite open paths.
By monotonicity of the process with respect to $p$ and $q$, there is a \emph{critical curve} $\gamma_k$ joining the points $(\frac{1}{d},0)$ and $(0,\frac{1}{d^k})$ which separates $\mathcal{N}_k$ and $\mathcal{P}_k$, as depicted in Figure~\ref{fig:phasespaceab}a.
Define
\begin{equation}
\label{eq:qcandpc}
q_c(p,k) = \inf\{q: (p,q)\in\mathcal{P}_k\}
\quad
\text{and}
\quad
p_c(q,k) = \inf\{p: (p,q)\in\mathcal{P}_k\}
.
\end{equation}

\begin{figure}[t]
\centering
\hspace*{\stretch{1}}
\includegraphics[width=.48\textwidth]{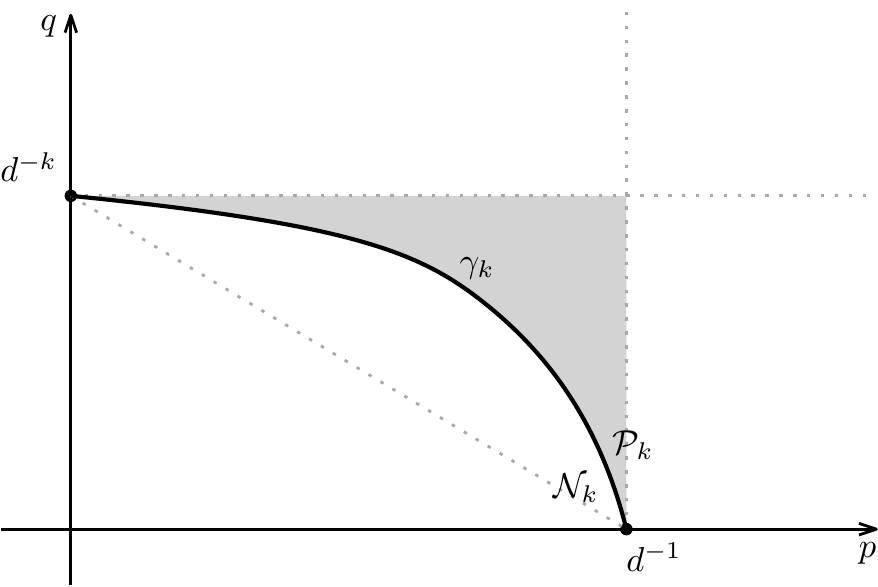}
\hspace*{\stretch{2}}
\includegraphics[width=.48\textwidth]{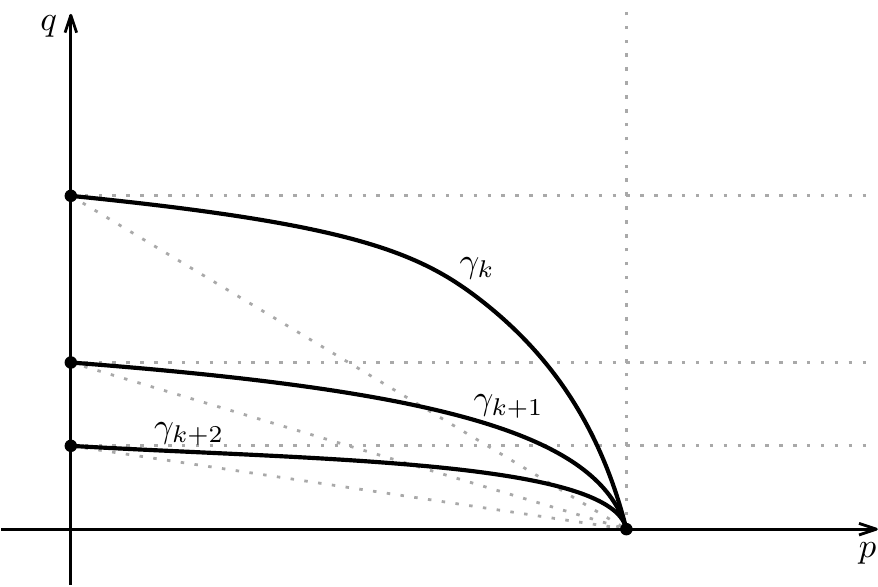}
\hspace*{\stretch{1}}
\\
\scriptsize
\hspace*{\stretch{1}}
(a)
\hspace*{\stretch{2}}
(b)
\hspace*{\stretch{1}}
\caption{Phase space and critical curve separating the percolative region $\mathcal{P}_k$ from the non-percolative region $\mathcal{N}_k$.
(a)~%
The curve stays between the three dotted lines.
In the gray region, infinite paths necessarily use both short and long edges.
(b)~%
Critical curves for different ranges $k$ meet at one point.
}
\label{fig:phasespaceab}
\end{figure}

Let $k$ be fixed. We show that $q_c$ is continuous and strictly decreasing in $p$ (equivalent formulations are that $p_c$ is strictly decreasing and continuous in $q$, that both $p_c$ and $q_c$ are continuous, or that $\gamma$ contains neither vertical nor horizontal segments).
In particular, $\gamma_k$ is described by $q=q_c(p,k)$ as well as by $p=p_c(q,k)$, and there is a non-trivial subregion of $\mathcal{P}_k$ at which infinite open paths necessarily use both long and short edges, see Figure~\ref{fig:phasespaceab}a.
A similar description is given in~\cite{IlievJansevanRensburgMadras15} for percolation with a defect plane.

We also show that $\gamma_{k+1}$ stays strictly below $\gamma_k$ for $p<d^{-1}$, and they meet only at the critical point $(d^{-1},0)$.
This means that $q_c(p,k)$ is strictly decreasing in $k$ for as long as it is positive, and analogously for $p_c(q,k)$, see Figure~\ref{fig:phasespaceab}b.

In \S\ref{sec:defresults}, we present the model and the above statements more formally.

In \S\ref{sec:pcqc}, we prove that $q_c(k,p)$ is continuous and decreasing in $p$.
In the proof,
we tile $\T_{d,k}$ by layers and consider a construction of the process where the state of tiles are sampled independently.
We then couple configurations with different values of $p$ and $q$ so that some advantage in $q$ compensates for small decreases in $p$ and vice-versa. Each comparison is done by finding one particular tile that makes no useful connections without extra open edges and at the same time makes all possible connections with their help. We learned this idea from~\cite{Teixeira06}.

In \S\ref{sec:hubs}, we show that $q_c(p,k+1)<q_c(p,k)$ for $p<d^{-1}$.
Together with the  results of \S\ref{sec:pcqc}, this inequality completes the previous description illustrated by Figure~\ref{fig:phasespaceab}b.
The proof involves a joint exploration of a percolation ``cluster'' in $\mathbb{T}_{d,k}$ and a percolation cluster in $\mathbb{T}_{d,k+1}$. The joint exploration is an algorithm in which parts of both clusters are revealed simultaneously using the same random variables.
After each step of the algorithm is concluded, there is an injective function from the revealed portion of the cluster in $\mathbb{T}_{d,k}$ to the one in $\mathbb{T}_{d,k+1}$.
When trying to ensure this, one might run into collisions, that is, situations where an edge that could potentially grow the cluster in $\mathbb{T}_{d,k}$ has as a counterpart an edge which does not grow the cluster in $\mathbb{T}_{d,k+1}$.
The challenge is thus to design the algorithm so that collisions do not occur.
We succeed in doing so by introducing a recursive procedure which alternately reveals clusters of short edges and then groups long edges, in a way that allows the comparison between the $k$ and $k+1$ scenarios.
This gives $q_c(p,k+1) \le q_c(p,k)$.
Strict inequality is obtained by extending the idea mentioned in the previous paragraph to a dynamic, hybrid construction.
When revealing the state of a whole batch of long edges at once we can use the increase in $k$ to compensate for a small decrease in $q$.

As a final remark, there seems to be no obvious way to adapt the argument just described to the graph with the vertices in $\mathbb{Z}^d$ and edges of the form $\{x,x\pm e_i\}$ and $\{x,x \pm k\cdot e_i\}$.
A similar joint exploration would lead to collisions as illustrated in Figure~\ref{fig:counter}.
Proving the inequality $p_c(k+1) \le p_c(k)$ mentioned in the previous footnote remains open, let alone strict inequality.

\begin{figure}[t]
\centering
\includegraphics[width=.8\textwidth]{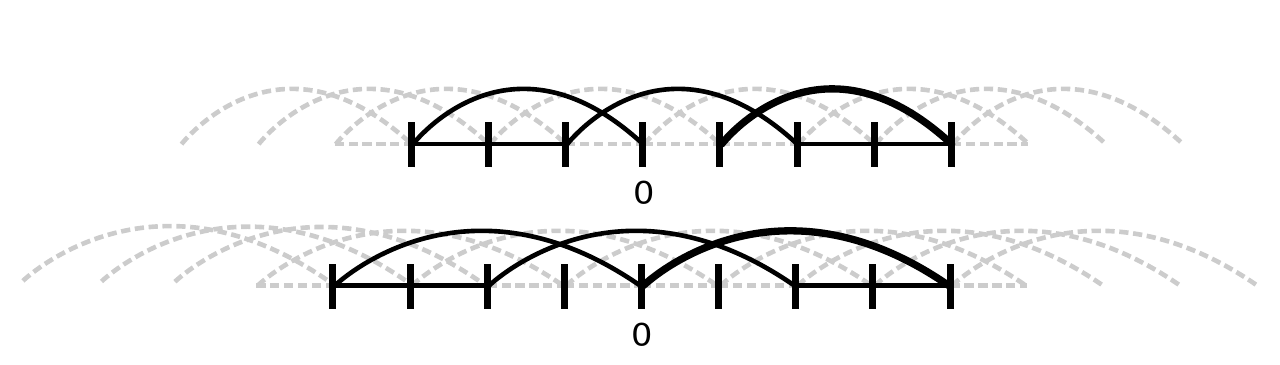}
\caption{Situation where the natural coupling of explorations which maps short edges to short edges and long edges to long edges leads to a ``collision'' in the graphs given by adding edges of length 3 and 4 to $\mathbb{Z}$. Dashed lines represent closed edges and full lines represent open edges. The bold edge being open increases the cluster of the first graph by one vertex, but has no effect on the cluster of the second graph.}
\label{fig:counter}
\end{figure}

\section{Definitions and results}
\label{sec:defresults}

Let $d\in\{2,3,4,\dots\}$ be fixed.
We denote $[d]= \{1,\ldots, d\}$, and we will make frequent use of the set
\begin{align*}
[d]_\star = \bigcup_{0\leq n < \infty} [d]^n;
\end{align*}
the set $[d]^0$ is understood to consist of a single point $o$.
Points of $[d]_\star\backslash \{o\}$ are represented as 
sequences $u = (u_1,\ldots, u_n)$. In case $u = (u_1,\ldots, u_m)$ and $v = (v_1,\ldots, v_n)$, we define the concatenation $u\cdot v = (u_1,\ldots, u_m, v_1,\ldots, v_n).$

Given $k \in \N = \{1,2,3,\dots\}$,
we define the oriented graph $\T_{d,k}$ as the graph with vertex set $\V_{d,k} = [d]_\star$ and edge set $\E_{d,k} = \E^\s_{d,k} \cup \E^\ell_{d,k}$, where 
\begin{align*}
&\E^\s_{d,k} = \{\langle u, u\cdot a\rangle: \;u \in \V_{d,k},\;a\in [d]\},
\\
&\E^\ell_{d,k} = \{\langle u, u\cdot r\rangle:\;u\in \V_{d,k},\; r\in [d]^k\}.
\end{align*}
These will be referred to as the sets of short and long edges of $\T_{d,k}$.

As the above notation suggests, we will normally use the letters $a,b$ for elements of $[d]$, the letters $r,s$ for elements of $[d]^k$ and the letters $u,v,w,x$ for general vertices of $\T_{d,k}$.

Consider the process in which, independently, short edges are open with probability $p$ and long edges are open with probability $q$. Let $\P_{p,q}$ denote the corresponding probability measure.

We define the event $u \leadsto v$ that there exist $u^0,u^1,\dots,u^{n-1},u^n$ such that $u^0=u$, $u^n=v$ and the edge
$\langle u^j,u^{j+1} \rangle$ is open for all $j<n$.
The event $u \leadsto \infty$ means that $u\leadsto v$ for infinitely many $v$.
Let
$\mathcal{P}_k = \{(p,q):\P_{p,q}(o\leadsto \infty) > 0\}$, $\mathcal{N}_k = [0,1]^2 \setminus \mathcal{P}_k$, and let $p_c(q,k)$ and $q_c(p,k)$ be given by~(\ref{eq:qcandpc}).

We prove the following monotonicity property.
\begin{theorem}
\label{prop:hub}
The inequality $q_c(p,k+1) < q_c(p,k)$
holds
unless
$q_c(p,k) = 0$.
\end{theorem}

This says that $\gamma_{k+1}$ stays under $\gamma_{k}$, and they can only intersect each other at the boundary $\{pq=0\}$, except maybe where one of them contains a vertical segment. The next result rules out the latter possibility, thus completing the picture provided in Figure~\ref{fig:phasespaceab}b.
\begin{theorem}
\label{thm:qccontinous} For each fixed $k\in\N$, the function $p \mapsto q_c(p,k)$ is continuous on $[0,1]$ and strictly decreasing on $[0,d^{-1}]$.
\end{theorem}

We observe that, as a consequence of the above results, defining
$$p_c(k) = \inf\{p: (p,p)\in\mathcal{P}_k\},$$
we have $p_c(k+1) < p_c(k)$, as the diagonal $\{(p,p):0\leq p \leq 1\}$ intersects the critical curves $\gamma_k$ at distinct points for different values of $k$.
However, for $k\ge 2$ this conclusion can be drawn from the simpler observation that the curves $\gamma_k$ are delimited by the dotted lines in Figure~\ref{fig:phasespaceab}b.

The next result says that there is no percolation along the critical curves $\gamma_k$.%

\begin{theorem}
\label{prop:qcnopercolation}
For $(p,q)$ on the critical curve $\gamma_k$, $\P_{p,q}(o\leadsto \infty)=0$.
\end{theorem} 
\begin{proof}
It is enough to prove that $\mathcal{P}_k$ is an open set in $[0,1]^2$. Define 
$$N_n = \#\left\{u \in [d]^{kn}:\; \text{there exists } v \in \cup_{i=0}^{k-1}[d]^i \text{ with }o \rightsquigarrow u\cdot v \right\},\quad n \in \mathbb{N}.$$ We claim that $N_n\rightarrow\infty$ a.s.\ on the event $o\leadsto \infty$.
Indeed, assuming $p<1$ and $q<1$, for each $j\in\N$ we have
$$
\P_{p,q} \left(\big. N_m = 0 \text{ for all } m > n \,\middle|\, N_1,\dots,N_n \right) > \sigma_j
$$
on the event that $N_n \le j$,
where $\sigma_j$ is a positive constant depending on $j$ and also on $p,q,d,k$, but not on $n$.  This shows that $\P_{p,q}(N_n = j \text{ i.o.})=0$ thus a.s.\ either $N_n\to 0$ or $N_n \to \infty$. The case $p=1$ or $q=1$ being trivial, the claim is proved.

Suppose $\theta_{p,q}:=\P_{p,q}(o\leadsto \infty)>0$ and let $\zeta<\theta_{p,q}$.
By the previous claim, there exists $n^*$ such that $\P_{p,q}(N_{n^*}> \frac{2 k^2}{\zeta} )>\zeta$.
Now observe that this probability is continuous in $(p,q)$, thus for $(p^\prime,q^\prime)$ close enough to $(p,q)$ it is still larger than $\zeta$.
From this observation, using the definition of $N_n$ and reverse union bound, there is $\ell \in \{kn^*,\dots,kn^*+k-1\}$ such that, with probability larger than $\frac{\zeta}{k}$, there are at least $\frac{2k}{\zeta}$ sites $u\in[d]^\ell$ such that $o \rightsquigarrow u$.


Therefore, the process $(N_{\ell i})_{i\in\N}$ dominates a supercritical branching process with offspring assuming values on $\{0,\lceil\frac{2k}{\zeta}\rceil\}$ and mean larger than $2$.
This implies that $\P_{p^\prime,q^\prime}(o \rightsquigarrow \infty)>0$, proving that $(p^\prime,q^\prime)\in\mathcal{P}_k$.
\end{proof}

\section{Long and short edge compensation}
\label{sec:pcqc}

The goal of this section is to prove Theorem~\ref{thm:qccontinous}. We will need the following elementary fact.

\begin{lemma}
\label{lem:enhance}
Let $P_\alpha$ denote probability measures on a given finite space $S$, parametrized by $\alpha\in[0,1]$, and such that $P_\alpha(x)$ is continuous in $\alpha$ for every $x\in S$.
Let $\kappa$ and $y$ be such that $P_{\kappa}(y)>0$.
Then for any $\alpha,\beta$ close enough to $\kappa$, there exists a coupling $(X,Y)$ such that
$X \sim P_{\alpha}$, $Y \sim P_{\beta}$ and such that, almost surely, $X=Y$ unless $X=y$ or $Y=y$.
\end{lemma}
\begin{proof}
Sample the pair $(X,Y)$ as
\[
(X,Y)=
\begin{cases}
(z,z) \ \text{ w.p. } \ P_{\alpha}(z) \wedge P_{\beta}(z), \\
(y,z) \ \text{ w.p. } \ [P_{\beta}(z) - P_{\alpha}(z)]^+, \\
(z,y) \ \text{ w.p. } \ [P_{\alpha}(z) - P_{\beta}(z)]^+, \\
\end{cases}
\]
for $z\ne y$,
and
\[
(X,Y) = (y,y) \ \text{ w.p. } \ 1 - \sum_{z\ne y}P_{\alpha}(z) \vee P_{\beta}(z)
.
\]
The last term is positive for $\alpha$ and $\beta$ are close to $\kappa$ because it is positive when $\alpha=\beta=\kappa$.
This sampling only include pairs for which $X=Y$ unless $X=y$ or $Y=y$.
From the first equation we have $\P(X=z)=P_\alpha(z)$ for all $z\ne y$, which all together imply $\P(X=y)=P_\alpha(y)$, and similarly for $Y$.
\end{proof}

We define the progeny of a vertex $u \in \V_{d,k}$ as the set
\[
\prog(u) = \{ u\cdot v \in \V_{d,k}:\; v \in [d]_\star\}
,
\]
i.e.\ it is the subtree started at $u$.
The progeny of an edge is defined as the progeny of its endpoint, that is, if $e = \langle u,v\rangle$, then $\prog(e) = \prog(v)$.

We now turn to the proof of Theorem~\ref{thm:qccontinous}.
Recall that $k$ is fixed, $q_c(0)=d^{-k}$ and $q_c(p)=0$ for $p>d^{-1}$.
Let $\mathscr{C}_{p,q,k}$ denote the percolation cluster of the root in $\mathbb{T}_{d,k}$ under the measure $\mathbb{P}_{p,q}$.
(We use the word ``cluster'' to denote the set of sites which can be reached from the root, so unlike unoriented percolation it does not define an equivalence class.)
We observe that, under this measure, the expected number of open edges having $o$ as an extremity is equal to $dp +d^k q$.
If such expectation is less than one, we can embed $\mathscr{C}_{p,q,k}$ in a subcritical branching process to conclude that $P_{p,q}(o \leadsto \infty) = 0$.
Therefore,
$ q_c(p,k) \geq d^{-k} - d^{-k+1} p. $
This implies that $q_c(p,k)>0$ for $p<d^{-1}$.
Since  $q_c(p,k) \leq q_c(0,k) = d^{-k}$, we also conclude that $p\mapsto q_c(p,k)$ is continuous at $p = 0$.

The proof of Theorem~\ref{thm:qccontinous}
will thus be complete once we establish the following two facts:
\begin{equation}\label{eq:cond_jump}
\begin{split}
&\text{for all } p_0, q, q' \in (0,1) \text{ with } q < q', \text{ there exist }p,p' \text{ with } p' < p_0 < p\\
 &\hspace{5cm}\text{ such that } \P_{p',q'}(o\leadsto \infty) \geq \P_{p,q}(o\leadsto \infty);
\end{split}
\end{equation}
\begin{equation}\label{eq:cond_jump2}
\begin{split}
&\text{for all } q_0, p, p' \in (0,1) \text{ with } p < p', \text{ there exist }q,q' \text{ with } q' < q_0 < q\\
 &\hspace{5cm}\text{ such that } \P_{p',q'}(o\leadsto \infty) \geq \P_{p,q}(o\leadsto \infty).
\end{split}
\end{equation}
Indeed, condition~\eqref{eq:cond_jump} rules out jump discontinuities in the curve $q = q_c(p,k)$ for $p > 0$, and condition~\eqref{eq:cond_jump2} rules out horizontal segments in this curve for $p < d^{-1}$.

We start the proof of~\eqref{eq:cond_jump} by introducing some notation. We let $\bar{\E}_{d,k} = \bar{\E}_{d,k}^\s \cup \bar{\E}_{d,k}^\ell$, where
\begin{align*}
&\bar{\E}_{d,k}^\s = \left\{e = \langle u, v\rangle \in \E^\s_{d,k}: u \in \cup_{n=0}^{2k-1}[d]^n\right\},
\\[.2cm]&\bar{\E}^\ell_{d,k} = \left\{e = \langle u, v\rangle \in \E^\ell_{d,k}: u \in \cup_{n=0}^{2k-1}[d]^n\right\}
.
\end{align*}
Configurations in $\bar{\Omega} = \bar{\Omega}_\s \times \bar{\Omega}_\ell = \{0,1\}^{\bar{\E}_{d,k}^\s \cup \bar{\E}_{d,k}^\ell}$ are written as $\bar{\omega} =(\uoms,\uoml)$. 

Given $A \subseteq \cup_{n=0}^{k-1} [d]^n$ and $\bar{\omega} = (\uoms,\uoml)$, we define
\begin{equation}
J_{\bar{\omega}}(A) = \bigcup_{u \in [d]^{2k}}\left\{\begin{array}{l}v \in \prog(u): \; \exists u^0, \ldots, u^n \in \V_{d,k} \text{ so that }u^0 \in A, \\  u^n = v \text{ and } \langle u^i, u^{i+1}\rangle \in \bar{\E}_{d,k},\; \bar{\omega}(\langle u^i,u^{i+1}\rangle) = 1 \; \forall i \end{array} \right\}.
\end{equation}
That is, $J_{\bar{\omega}}(A)$ is the set of vertices in $\cup_{u\in[d]^{2k}} \prog(u)$ that are reachable by paths started from $A$ and consisting only of open edges of $\bar{\E}_{d,k}$. Note that in such a path, all edges have both extremities in $\cup_{n=0}^{2k-1}[d]^n$ except for the last one, which has only one extremity in $\cup_{n=0}^{2k-1}[d]^n$.
In particular, $J_{\bar{\omega}}(A) \subseteq \cup_{n=2k}^{3k-1}[d]^n$. 

Now, define the deterministic configurations $\bar{\omega}^*_{\s} \in \bar{\Omega}_\s$ and $\bar{\omega}^*_{\ell,1},\bar{\omega}^*_{\ell,2}\in\bar{\Omega}_\ell$ by setting 
\begin{align*}&\bar{\omega}^*_{\ell,1}\equiv 0,\quad \bar{\omega}^*_{\ell,2} \equiv 1 \quad \text{ and }\quad\bar{\omega}^*_{\s}(\langle u,v\rangle) = 1 \text{ if and only if } u\notin [d]^{2k-1}.
\end{align*}
Let $0<p_0<1$ and $0<q<q'<1$.
By Lemma~\ref{lem:enhance}, if $p$ and $p'$ with $p' < p_0 < p$ are chosen sufficiently close to $p_0$, then there exists a coupling of configurations
$$X = (X_\s,X_{\ell,1},X_{\ell,2}) \quad \text{ and } \quad Y= (Y_\s, Y_{\ell,1},Y_{\ell,2})$$
in $\bar{\Omega}_\s\times \bar{\Omega}_\ell\times \bar{\Omega}_\ell$ so that the following holds:
\begin{itemize}
\item
the values of $X_\s$, $X_{\ell,1}$ and $X_{\ell_2}$ in all edges are independent;
\item
$X_\s$, $X_{\ell,1}$ and $X_{\ell,2}$ assign each edge to be open with respective probabilities $p$, $q$ and $\frac{q'-q}{1-q}$;
\item
the values of $Y_\s$, $Y_{\ell,1}$ and $Y_{\ell_2}$ in all edges are independent;
\item
$Y_\s$, $Y_{\ell,1}$ and $Y_{\ell,2}$ assign each edge to be open with respective probabilities $p'$, $q$ and $\frac{q'-q}{1-q}$;
\item
the following event has probability one: \begin{equation}\label{eq:3_casos}
\{X = Y\} \cup \{X = (\bar{\omega}^*_{\s},\bar{\omega}^*_{\ell,1},\bar{\omega}^*_{\ell,2})\} \cup \{Y = (\bar{\omega}^*_{\s},\bar{\omega}^*_{\ell,1},\bar{\omega}^*_{\ell,2})\}.
\end{equation}
\end{itemize}

Now take $\bar{\omega}_\s = X_\s$, $\bar{\omega}_\ell = X_{\ell,1}$, $\bar{\omega}'_\s = Y_\s$, $\bar{\omega}'_\ell = Y_{\ell,1} \vee Y_{\ell,2}$.

The main observation is that each of the three events in~\eqref{eq:3_casos} implies that, for every $A\subseteq \cup_{n=0}^{k-1}[d]^n$,
\begin{equation}\label{eq:a_ver}
J_{\bar{\omega}}(A) \subseteq J_{\bar{\omega}'}(A).
\end{equation}
Indeed, on the first event we have $\bar{\omega}' \geq \bar{\omega}$, on the second event we have $J_{\bar{\omega}}(A)=\emptyset$, and on the third event $J_{\bar{\omega}'}(A)$ contains the set of sites $y\in\cup_{n=2k}^{3k-1}[d]^n$ that are in $\prog(x)$ for some $x\in A$, which always contains $J_{\bar{\omega}}(A)$.

Finally, with this coupling at hand, we can sample configurations $\omega,\omega' \in \{0,1\}^{\E_{d,k}}$ such that the restrictions of $\omega$ and $\omega'$ to sets of the form
\[
\left\{\langle u\cdot v, w \rangle \in \E_{d,k}: v \in \cup_{n=0}^{2k-1}[d]^n\right\}
\]
with $u \in \cup_{m\in 2\N} [d]^{mk}$ are independent and sampled from the (appropriately translated) coupling measure.
Then $\omega$ and $\omega'$ are distributed as $\P_{p,q}$ and $\P_{p',q'}$ respectively, and the cluster of the root in $\omega$ is a subset of the cluster of the root in $\omega'$.
This concludes the proof of~\eqref{eq:cond_jump}.

We now turn to the proof of \eqref{eq:cond_jump2}.
As the two proofs are very similar, we now only outline the main steps of the argument.

We let $\bar{\E}^\s_{d,k}$, $\bar{\E}^\ell_{d,k}$, $\bar{\E}_{d,k}$, $\bar{\Omega}_\s$, $\bar{\Omega}_\ell$ and $J_{\bar{\omega}}(A)$ be the same as before. A special configuration $\bar{\omega}^* \in \bar{\Omega}_\s \times \bar{\Omega}_\s \times \bar{\Omega}_\ell$ is defined as follows:
$$\bar{\omega}^*_{\s,1} \equiv 0,\quad \bar{\omega}^*_{\s,2} \equiv 1,\quad \bar{\omega}^*_\ell(\langle r,s\rangle) = 1 \text{ if and only if } r \in \cup_{n=k}^{2k-1} [d]^n.$$

Using Lemma~\ref{lem:enhance}, we obtain $q'< q_0 < q$ and a coupling of $X = (X_{\s,1},X_{\s,2},X_\ell)$ and $Y=(Y_{\s,1},Y_{\s,2}, Y_\ell)$ so that the following hold.
The values of $X_{\s,1}$, $X_{\s,2}$ and $X_{\ell}$ in all edges are independent;
$X_{\s,1}$, $X_{\s,2}$ and $X_{\ell}$ assign each edge to be open with respective probabilities $p$, $\frac{p'-p}{1-p}$ and $q$;
the values of $Y_{\s,1}$, $Y_{\s,2}$ and $Y_{\ell}$ in all edges are independent;
$Y_{\s,1}$, $Y_{\s,2}$ and $Y_{\ell}$ assign each edge to be open with respective probabilities $p$, $\frac{p'-p}{1-p}$ and $q'$;
the following event has probability one:
\begin{equation*}
\{X = Y\} \cup \{X = (\bar{\omega}^*_{\s,1},\bar{\omega}^*_{\s,2},\bar{\omega}^*_{\ell})\} \cup \{Y = (\bar{\omega}^*_{\s,1},\bar{\omega}^*_{\s,2},\bar{\omega}^*_{\ell})\}.
\end{equation*}
We then let $\bar{\omega}_\s = X_{\s,1}$, $\bar{\omega}_\ell = X_\ell$, $\bar{\omega}'_{\s} = Y_{\s,1} \vee Y_{\s,2}$ and $\bar{\omega}'_\ell = Y_\ell$.
This coupling then guarantees~\eqref{eq:a_ver} as before, which concludes the proof of Theorem~\ref{thm:qccontinous}.

\section{Comparison of different ranges}
\label{sec:hubs}

In this section we prove Theorem~\ref{prop:hub}.
The general idea behind the proof is to explore short edges until reaching a dead end, then use a coupling construction to show that one has a better chance to proceed from each dead end when $k$ is larger.

Let $u \in \V_{d,k}$ and $r = (r_1,\ldots, r_k) \in[d]^k$, so that $e= \langle u, u\cdot r\rangle \in \E^\ell_{d,k}$. We define the trace of $e$ to be the set of short edges
$$\text{trace}(e) = \{\langle u, u\cdot r_1\rangle, \langle u\cdot r_1,u\cdot (r_1,r_2)\rangle,\ldots, \langle u\cdot (r_1,\ldots, r_{k-1}),u\cdot r\rangle\}.$$

Fix $\omega = (\omega_\s,\omega_\ell)$, with $\omega_\s \in \{0,1\}^{\E^\s_{d,k}}$ and $\omega_\ell \in \{0,1\}^{\E^\ell_{d,k}}$, and a set $A \subset \V_{d,k}$.
We let $\Pi(A)$ be the cluster of $A$ in $\omega$,
that is, the set of vertices of $\T_{d,k}$ which can be reached by a path started from some vertex of $A$ and consisting of directed edges which are open in $\omega$ (note that $\Pi(A)$ depends on $A$ and $\omega$ but we omit $\omega$ from the notation; this will also be the case for further notation that we introduce). We also let $\pi(A)$ be the cluster of $A$ in $\omega_\s$, that is, the set of vertices of $\T_{d,k}$ that can be reached by a path started from some vertex of $A$ and consisting of \textit{short} edges, all of which are open in $\omega_\s$. Note that $A \subseteq \pi(A) \subseteq \Pi(A)$.

We say a short edge $e = \langle u, v\rangle \in \E^\s_{d,k}$ is a \emph{hub} for $A$ (in $\omega$) if the following two conditions hold:
\begin{equation}\label{eq:def_prog}\prog(v) \cap \pi(A) = \varnothing \qquad \text{and} \qquad
\prog(u) \cap \pi(A) \neq \varnothing.\end{equation}
We let $\sigma(A)$ denote the set of hubs for $A$ in $\omega$.

\begin{lemma}
\label{lem:pisigma}
Let $\omega \in \{0,1\}^{\E_{d,k}}$ and $A \subset \V_{d,k}$. Then,
\begin{equation}
\label{eq:prog_are_disjoint}
\text{the progenies }\prog(e)\text{ for }e\in \sigma(A) \text{ are disjoint.}
\end{equation}
Further assuming that
\begin{equation}
\text{there exists } w \in \V_{d,k} \text{ such that } A \subset \left\{ w \cdot v: v \in \cup_{n=0}^{k} [d]^n\right\}, \label{eq:assump_lem}
\end{equation}
we also have
 \begin{equation}\label{eq:prop_of_progs}\begin{split}\text{for any } e = \langle u,u\cdot r\rangle \in \E^\ell_{d,k}\text{ such that }u \in \pi(A)\text{ and }u\cdot r \notin \pi(A),\\\text{ there exists a unique }e'\in \text{trace}(e) \cap \sigma(A)\end{split}
 \end{equation}
 and
 \begin{equation}\label{eq:prop2_of_progs}\begin{split}
 \Pi(A) \text{ is the disjoint union of } \pi(A) \text{ and the sets }\\ \Pi(A) \cap \prog(e) \text{ for } e \in \sigma(A).
\end{split}\end{equation}
\end{lemma}
\begin{proof}
To prove~(\ref{eq:prog_are_disjoint}), assume that there are two distinct hubs
\[
e = \langle u, v\rangle,\;e' = \langle u', v'\rangle \in \sigma(A):\quad\prog(e) \cap \prog(e') \neq \varnothing
.
\]
Then either $u \in \prog(v')$ or $u' \in \prog(v)$.
Without loss of generality we assume the latter. Together with~\eqref{eq:def_prog} applied to $e'$, this implies that $$\prog(v) \cap \pi(A) \supset \prog(u') \cap \pi(A) \neq \varnothing,$$ which contradicts~\eqref{eq:def_prog} applied to $e$.

Now fix an edge $e = \langle u,u\cdot r\rangle$ as in~\eqref{eq:prop_of_progs}. Consider the $k$ short edges in the trace of $e$. By the first statement, we know that at most one of these short edges is in $\sigma(A)$. In order to show that one of them is in $\sigma(A)$, it suffices to show that
\begin{equation}
\label{eq:suffices}
\prog(u) \cap \pi(A) \neq \varnothing\qquad \text{ and } \qquad \prog(u\cdot r) \cap \pi(A) = \varnothing
.
\end{equation}
The first claim of~\eqref{eq:suffices} follows from the fact that $u \in \pi(A)$; let us prove the second.
We are given that $u \cdot r \notin \pi(A)$, so it suffices to prove that $\prog(u\cdot r) \cap A = \varnothing$. For vertices $u', v'$ with $v' \in \prog(u')$, let $\text{dist}(u',v')$ denote the length of the unique path of short edges from $u'$ to $v'$.
Then, \eqref{eq:assump_lem} gives $\text{dist}(w,v) \leq k$ for all $v \in A$. If $v \in \prog(u\cdot r)$ and $v \neq u \cdot r$, then
$$\text{dist}(w,v) > \text{dist}(w,u\cdot r) = \text{dist}(w,u)+k,$$
so $v \notin A$.
We also have $u\cdot r \notin A$, so the proof of \eqref{eq:suffices} is complete.

Statement~\eqref{eq:prop2_of_progs} is an immediate consequence of~\eqref{eq:prog_are_disjoint} and~\eqref{eq:prop_of_progs}.
\end{proof}

Again fix $\omega \in \{0,1\}^{\E_{d,k}}$ and $A \subset \V_{d,k}$ satisfying~\eqref{eq:assump_lem}.
For each hub $e \in \sigma(A)$, we define 
\begin{align*}
&
R(A,e) = \{e'=\langle u',v'\rangle \in \E_{d,k}^\ell:\;u'\in\pi(A) \text{ and } e\in \text{trace}(e')\}
,\\[.2cm]
&
\bar{S}(A,e) = \{v' \in \V_{d,k}: \langle u',v'\rangle \in R(A,e) \text{ for some }u' \in \V_{d,k}\}
,\\[.2cm]
&
S(A,e) = \{v' \in \V_{d,k}: \omega(\langle u',v'\rangle) = 1 \text{ for some }  \langle u',v'\rangle \in R(A,e)\}.
\end{align*} 

Note that $S(A,e) \subseteq \bar{S}(A,e) \subseteq \prog(e)$.
Also note that, if $e_1, e_2 \in \sigma(A)$ are distinct, then $R(A,e_1)$ and $R(A,e_2)$ are disjoint, by~\eqref{eq:prog_are_disjoint}.
Finally, note that for every $e \in \sigma(A)$, we have
\begin{equation*}
\Pi(A) \cap \prog(e) = \Pi(S(A,e)),
\end{equation*}
so that~\eqref{eq:prop2_of_progs} can be restated as
\begin{equation}\label{eq:restate}
\Pi(A) = \pi(A) \cup \left(\cup_{e \in \sigma(A)} \Pi(S(A,e))\right),
\end{equation}
where the union is disjoint.

For $A$ satisfying~\eqref{eq:assump_lem}, we now let $\mathscr{C}_{p,q,k}(A)$ be the random set $\Pi(A)$ when $\omega$ is sampled from the measure $\mathbb{P}_{p,q}$ on percolation configurations on $\mathbb{T}_{d,k}$. Note that $\mathscr{C}_{p,q,k} = \mathscr{C}_{p,q,k}(\{o\})$.

We observe that, conditioning on $\pi(A)$, $\sigma(A)$ is determined and the sets $\Pi(S(A,e))$ are independent over $e \in \sigma(A)$.
Indeed, $\Pi(S(A,e))$ is determined by $\pi(A)$ and $\omega(e')$ for all
$$e' = \langle u', v'\rangle \text{ with } v' \in \prog(e).$$

The sets of edges displayed above are disjoint for distinct choices of $e \in \sigma(A)$.
Indeed, assume $e, f \in \sigma(A)$, $e \neq f$, and $e' = \langle u', v'\rangle$, $f' = \langle w', x' \rangle$ are long edges with $v' \in \text{prog}(e)$, $x'\in \text{prog}(f)$. Then, since~\eqref{eq:prog_are_disjoint} gives $\text{prog}(e) \cap \text{prog}(f) = \varnothing$, we obtain $v' \neq x'$, so $e' \neq f'$.

Guided by this consideration, we now present a recursive exploration algorithm to reveal $\mathscr{C}_{p,q,k}(A)$. The algorithm starts by applying the following two steps to the set $A$:

\emph{Step~1.}
Explore $\pi(A)$ by revealing only the edges in $\omega_\s$ that are necessary.
More precisely, grow $\pi(A)$ progressively by starting from $A$ and querying the open/closed-state of short edges one by one, each time selecting a short edge $e = \langle u, v \rangle$ such that $u$ is already included in $\pi(A)$ and $v$ is not (and also following some lexicographic-type priority rule that guarantees that the full $\pi(A)$ is explored). Note that this also determines $\sigma(A)$, hence $\bar{S}(A,e)$ for each $e \in \sigma(A)$.

\emph{Step~2.}
For each $e \in \sigma(A)$, reveal $S(A,e)$. This is the same as revealing the value of $\omega_\ell(e')$ for each long edge $e'\in R(A,e)$.

Note that, if $e = \langle u, v\rangle \in \sigma(A)$, then $S(A,e) \subseteq \{v \cdot w: w\in \cup_{n=0}^{k-1} [d]^n \}$, so that property~\eqref{eq:assump_lem} holds with $A$ replaced by $S(A,e)$.
The algorithm then proceeds by applying Steps~1 and~2 to each of the sets $S(A,e)$, which take the role of $A$. That is: in Step~1 it explores $\pi(S(A,e))$, which also reveals $\sigma(S(A,e))$, and in Step~2, for each $e' \in \sigma(S(A,e))$, it reveals $S(S(A,e),e')$. The recursion then continues to further levels. By~\eqref{eq:restate}, this reveals the whole cluster $\mathscr{C}_{p,q,k}(A)$.
 
We now want to look at the distributions of $S(A,e)$ and $\Pi(S(A,e))$ for $e \in \sigma(A)$. Although these distributions are easily understood, they are somewhat clumsy to describe, so we will need some more notation. 

First, fix $e = \langle u,v \rangle \in \sigma(A)$ with $v = (v_1,\ldots, v_n)$. Define
$$ \beta(A,e) =\{ i \in \{1,\ldots, k\} : (v_1,\ldots, v_{n-i}) \in \pi(A)\},$$
it describes which ancestors of $e$ have been reached from $A$ using short edges and could reach $\prog(e)$ using long edges (the $\omega$-state of which is not looked at).
Note that
$$R(A,e) = \{\langle (v_1,\ldots, v_{n-i}), v\cdot w\rangle: i \in \beta(A,e),\; w \in [d]^{k-i}\},$$
so that
$$\bar{S}(A,e) = \{v\cdot w: i\in\beta(A,e),\;w \in [d]^{k-i}\}.$$

Second, we define some shift mappings in $\T_{d,k}$.
Given $u \in \V_{d,k}$, we let $\uptau_u: \prog(u) \to \V_{d,k}$ be the function 
$$\uptau_u(u \cdot v) = v,\quad v \in [d]_\star.$$
If $e = \langle u,v\rangle \in \E^\s_{d,k}$, we let $\uptau_e = \uptau_v$.

Third, given $b \subset \{1,\ldots, k\}$, we let $\mathscr{A}_{q,k}(b)$ denote the distribution of the random subset of $\cup_{i \in b} [d]^{k-i} $
in which, independently, each point is included with probability $q$.

Let $A \subseteq \V_{d,k}$ satisfy~\eqref{eq:assump_lem}.
Conditioning on $\pi(A)$, for each $e \in \sigma(A)$ we have
\begin{align}
\label{eq:lawofS}\uptau_e(S(A,e)) \stackrel{(d)}{=} \mathscr{A}_{q,k}(\beta(A,e))
\end{align}
and the law of $\uptau_e(\Pi(S(A,e)))$ is equal to the law of the cluster of $B$ in $\mathbb{T}_{d,k}$, where $B$ is chosen according to $\mathscr{A}_{q,k}(\beta(A,e))$.

We finally turn to the desired comparison between $\mathscr{C}_{p,q,k}$ for different values of the parameters. Given $A, B \subseteq \V_{d,k}$, let us write $A \preceq B$ in case there exist $u, v \in [d]_\star$ such that $A \subseteq \prog(u)$ and $\uptau_u(A) \subseteq \uptau_v(B \cap \prog(v))$.
\begin{lemma}
\label{lem:comparison_l}
For any $k \in \N$ and $q \in (0,1)$, there exists $q'<q$ such that the following holds. Let $b' \subseteq \{1,\ldots, k+1\}$ and $b = b' \cap \{1,\ldots, k\}$. There exists a coupling $(A,B)$ of random sets $A, B \subseteq [d]_\star$ such that
$$A \preceq B,\quad A \stackrel{(d)}{=} \mathscr{A}_{q,k}(b) \quad\text{ and }\quad B \stackrel{(d)}{=} \mathscr{A}_{q',k+1}(b').$$ 
\end{lemma}

With this lemma at hand, we are ready to conclude the proof of Theorem~\ref{prop:hub}.
Fix $p,q\in (0,1)$ and $k \in \mathbb{N}$, and choose $q'$ corresponding to $k$ and $q$ in Lemma~\ref{lem:comparison_l}. The idea is to compare the explorations of $\mathscr{C}_{p,q,k}$ and $\mathscr{C}_{p,q',k+1}$ using coupling. Recall that our algorithm to explore a cluster proceeds by the iterative application of two steps.
Step~1 grows a portion of the cluster using only short edges, so it can be taken as the same for both explorations, since short edges have the same probability of being open in both.
Step~2 inspects ``exit routes", using long edges, from the portion of cluster revealed in Step~1; Lemma~\ref{lem:comparison_l} guarantees that this is better (in the sense of $\preceq$-domination) for $\mathscr{C}_{p,q',k+1}$ than for $\mathscr{C}_{p,q,k}$.

Let us now present the coupling of explorations more formally.
Note that we are dealing with percolation in the two graphs $\mathbb{T}_{d,k}$ and $\mathbb{T}_{d,k+1}$ simultaneously; these graphs have the same set of vertices (namely, $[d]_\star$) and same set of short edges, but the long edges differ.
A set $A \subset [d]_\star$ satisfying condition \eqref{eq:assump_lem} for $k$ also satisfies it when $k$ is replaced by $k+1$.
For such a set, and for $e \in \sigma(A)$, instead of $S(A,e)$ we will now write $S_k(A,e)$ and $S_{k+1}(A,e)$ to distinguish this set in the two graphs.

The coupled exploration of $\mathscr{C}_{p,q,k}$ and $\mathscr{C}_{p,q',k+1}$ starts with revealing $\pi(\{o\})$, which we can take as the same in both clusters. Thus, $\sigma(\{o\})$ is also the same in both graphs, and we enumerate
$$\sigma(\{o\}) = \{e^1,\ldots, e^N\}.$$
Also write
$$A^i = S_k(\{o\},e^i),\qquad B^i = S_{k+1}(\{o\},e^i),\qquad i=1,\ldots, N.$$
By Lemma~\ref{lem:comparison_l}, these can be sampled with $A^i \preceq B^i$, so there exist $u^i,v^i \in [d]_\star$ and $\tilde{B}^i \subseteq B^i$ such that $$A^i \subset \text{prog}(u^i),\; \tilde{B}^i \subset \text{prog}(v^i),\;\uptau_{u^i}(A^i) = \uptau_{v^i}(\tilde{B}^i).$$
The second level of the exploration then proceeds as follows.
For each $i \in \{1,\ldots, N\}$, take $\pi(\uptau_{u^i}(A_i))$ and $\pi(\uptau_{v^i}(\tilde{B}^i))$ as the same in both clusters, enumerate
$$\sigma(\uptau_{u^i}(A^i)) = \sigma(\uptau_{v^i}(\tilde{B}^i)) =  \{e^{i,1},\ldots, e^{i,N_i}\},$$
and let
$$A^{i,j} = S_k(\uptau_{u^i}(A^i),e^j),\;B^{i,j} = S_{k+1}(\uptau_{v^i}(\tilde{B}^i),e^j),\qquad j = 1,\ldots, N_i,$$
which can be sampled with $A^{i,j} \preceq B^{i,j}$ for each $j$. Further levels are then carried out in the same way. The construction guarantees that $\mathscr{C}_{p,q,k}$ is embedded in $\mathscr{C}_{p,q',k+1}$, concluding the proof of Theorem~\ref{prop:hub}. It remains only to prove the previous lemma.

\begin{proof}
[Proof of Lemma~\ref{lem:comparison_l}]
We can assume that $k +1 \notin b'$, so that $b = b'$. In that case, for $\hat{q} \in (0,1)$ and $\hat{B}$ a random subset of $[d]_\star$,
\begin{equation}
\hat{B} \sim \mathscr{A}_{\hat{q},k+1}(b) \quad \text{if and only if} \quad
\begin{array}{l}\uptau_{(1)}(\hat{B}),\ldots, \uptau_{(d)}(\hat{B}) \text{ i.i.d. }\\\text{ and distributed as } \mathscr{A}_{\hat{q},k}(b).
\end{array}
\end{equation}
We now define sets $S_1^*, \ldots, S_d^* \subset [d]_\star$ by
$$
S_1^* = \varnothing, \qquad S_2^*,\ldots, S_d^* = \bigcup_{i \in b}\; [d]^{k-i}
.
$$
By Lemma~\ref{lem:enhance}, there exists $q' < q$ and a coupling of random sets $X_1,\ldots, X_d$, $Y_1,\ldots, Y_d \subset [d]_\star$ so that $X_1,\ldots, X_d$ are independent and distributed as $\mathscr{A}_{q,k}(b)$, $Y_1,\ldots, Y_d$ are independent and distributed as $\mathscr{A}_{q',k}(b)$ and the following event has probability 1:
\begin{equation}\begin{split}
&\{(X_1,\ldots, X_d) = (Y_1,\ldots, Y_d)\} \cup \{(X_1,\ldots, X_d) = (S_1^*,\ldots, S_d^*)\} \\&\hspace{6cm}\cup \{ (Y_1,\ldots, Y_d) = (S_1^*,\ldots, S_d^*)\}.\end{split}
\end{equation}
The desired conclusion now follows by setting
\[
A = X_1,\qquad B = \cup_{a \in [d]} \{a\cdot u: u \in Y_a \}
.
\qedhere
\]
\end{proof}

\section*{Acknowledgements}

The authors would like to thank Aernout van Enter for helpful discussions, and Gábor Pete for pointing out an inaccurate citation in an earlier version of this paper. B.N.B.L. would like to thank the University of Groningen and D.V. would like to thank NYU-Shanghai for support and hospitality.
This project was supported by grants
CNPq 309468/2014-0, FAPEMIG (Programa Pesquisador Mineiro),
PIP 11220130100521CO,
PICT-2015-3154, PICT-2013-2137, PICT-2012-2744,
Conicet-45955 and MinCyT-BR-13/14.

\renewcommand{\baselinestretch}{1}
\setlength{\parskip}{0pt}
\small

\bibliographystyle{bib/leo}
\bibliography{bib/leo}

\end{document}